\documentclass{amsart}
\usepackage{amsmath}
\usepackage{amsmath}
\usepackage[T1]{fontenc}
\usepackage{textcomp} 
\DeclareEncodingSubset{TS1}{hlh}{1} 
\usepackage{amscd,amssymb}
\usepackage{psfrag}
\usepackage{graphicx}
\usepackage[dvipsnames]{xcolor}
\usepackage[all]{xy}
\SelectTips{cm}{}
\allowdisplaybreaks

\numberwithin{equation}{subsection}

\newtheorem{propo}{Proposition}[section]

\newtheorem{theor}[propo]{Theorem}

\theoremstyle{definition}

\theoremstyle{remark}

\let\oldmarginpar\marginpar
\renewcommand\marginpar[1]{\oldmarginpar{\footnotesize #1}}

\newcommand{\Int}{\operatorname{Int}}

\newcommand{\card}{\operatorname{card}}

\newcommand{\sign} {\operatorname {sign}}

\newcommand{\wri}{\operatorname{wri}}
\newcommand{\Wri}{\operatorname{Wri}}
\newcommand{\co}{\operatorname{co}}
\newcommand{\cro}{\operatorname{cro}}
\newcommand{\spun}{\operatorname{spun}}

\newcommand{\lk}{\operatorname{lk}}
\newcommand{\mi}{\operatorname{mir}}

\newcommand{\ridge}{\operatorname{rid}}

\newcommand{\LK}{\operatorname{LK}}

\newcommand{\rsdraw}[3]{\raisebox{-#1\height}{\scalebox{#2}{\includegraphics{#3.eps}}}}

\begin{document}

\title[Knots and links in 2-complexes]{Knots and links in 2-complexes}

    \author[Vladimir Turaev]{Vladimir Turaev}
    \address{
    Vladimir Turaev \newline
    \indent   Department of Mathematics \newline
    \indent  Indiana University \newline
    \indent Bloomington IN47405, USA\newline
    \indent $\mathtt{vturaev@yahoo.com}$}

\begin{abstract}  We introduce and study knots and links in 2-dimensional complexes. In particular, we  define linking numbers for  oriented two-component links in 2-complexes and  a Kauffman-type bracket polynomial for   links in 2-complexes. We also discuss relationships  with  knots and links in 3-manifolds.
\end{abstract}

\maketitle

\section {Introduction} 

Classical knot theory
studies isotopy classes of embedded circles in  3-dimensional Euclidean space $\Bbb  R^3$.  This theory
extends in many directions including  knots in arbitrary 3-manifolds, virtual knots, and high-dimensional knots. 
To present a  knot in $\Bbb  R^3$ one  can draw its generic projection to the plane and indicate at all crossings which of the two branches of the knot lies above the other branch. The resulting plane pictures  are called  knot diagrams. Knot isotopies can be formulated   as sequences of Reidemeister moves which are certain standard  local modifications of  knot diagrams. Instead of the plane $\Bbb R^2$ one can equivalently use  a disc in $\Bbb R^2$ or the 2-sphere  $S^2=\Bbb R^2\cup\{\infty\}$. 

In this paper, we introduce a version of  knot theory where the 2-sphere  is replaced by an arbitrary compact 2-dimensional complex~$X$.  By a knot in~$X$ we shall mean a generic loop  in~$X$ provided with  certain additional data. 
Note that~$X$ is not supposed to be embedded (or even embeddable) in a  3-manifold,  so that we cannot apply here the usual language of over/under-going branches. Our  additional data    is  formulated  in terms of the complements of   loops in~$X$. Similar ideas allow us to define links in~$X$.  
We define isotopy of such links using appropriate versions of the Reidemeister moves. The objective of this theory is to study isotopy classes of knots and links in~$X$. 

We do not meet in this setting any plausible analogues of   standard notions of knot theory such as the knot exterior, the knot group, the Seifert form, the Alexander  polynomial, etc. 
On the other hand, our knots and links in a 2-complex $X$ do share  a number of features of knots and links in 3-manifolds. First, we can draw  usual planar
link diagrams in regions of~$X$ so that classical links determine  links in~$X$.
 Second, when~$X$ is embedded in a 3-manifold, every link in~$X$ determines a link in this 3-manifold.  Third, several fundamental  invariants of classical links  extend to links in  2-complexes.   We will demonstrate it  for  the linking number and  the Kauffman bracket polynomial. In a separate paper we will discuss such extensions for the Khovanov homology.  Also,
we can define   skein modules of  2-complexes analogous to skein modules of 3-manifolds. 

To finish the Introduction, I would like to heartily thank Alexis Virelizier, who did for me all the computer pictures in this paper. I am also thankful to C. Livingston and A. Sikora for helpful remarks on this paper.

\section{Preliminaries on  classical  links}\label{class}  

 We introduce  dotted  diagrams of classical knots and links.

\subsection{Basics}\label{basics} Diagrams of a  link in $\Bbb R^3$ are obtained by projecting the  link to the plane $\Bbb R^2$ and indicating the over/under-passing branches at all crossings. The link should be preliminary deformed into a generic position  to ensure that all crossings in its projection are double and transversal. Looking at the link diagram, we can immediately recover the link up to ambient isotopy in $\Bbb R^3$. The Reidemeister theorem says that two  link diagrams  represent isotopic links in  $\Bbb R^3$ if and only if these  diagrams can be obtained from each other by  a finite sequence of ambient isotopies in the plane,  Reidemeister moves $\Omega_1^\pm, \Omega_2, \Omega_3$ shown in Figure~1, and the inverse moves. 

\begin{figure}[h]
\begin{center}
\psfrag{X}[Bc][Bc]{\scalebox{1}{$\Omega_1^-$}}
\psfrag{A}[Bc][Bc]{\scalebox{1}{$\Omega_1^+$}}
\rsdraw{.45}{.9}{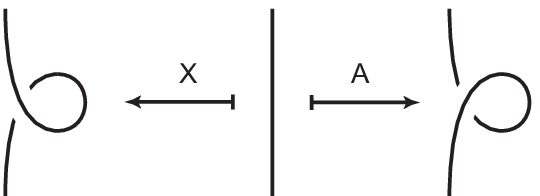}\\[1.5em]
\psfrag{D}[Bc][Bc]{\scalebox{1}{$\Omega_2$}}
\psfrag{T}[Bc][Bc]{\scalebox{1}{$\Omega_3$}}
\rsdraw{.45}{.9}{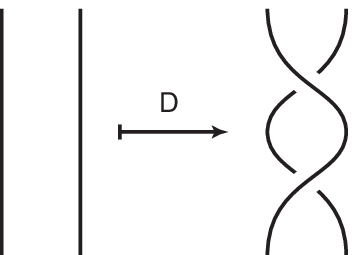}\hspace{1.6cm} \rsdraw{.45}{.9}{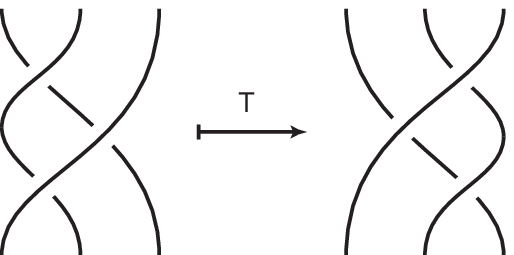}
  \end{center}
\caption{The Reidemeister moves}
\label{fig-Reidemeister}
\end{figure}

\subsection{Dotted diagrams}\label{dottedbasics} Instead of traditional link diagrams  we will use plane  pictures which we call \emph{dotted  link diagrams}. Observe  that near every crossing, the (generic) link projection in the plane looks like two transversal lines  splitting the plane into four areas surrounding the crossing. Pick the overpassing line and rotate it counterclockwise sweeping two of these areas. Put a dot in each of those two areas, see Figure~2. These dots allow us to  recover the overpassing line: moving the dots clockwise around the crossing we hit the overpassing line before hitting  the underpassing line. In this way, the  language of link diagrams with over/under-passes
 can be translated in the language of dotted  diagrams and vice versa. 
 For example, the dotted  diagrams in Figure~3 represent a left-handed trefoil and a right-handed trefoil.

\begin{figure}[h]
\begin{center}
\rsdraw{.45}{.9}{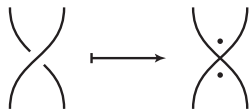}
  \end{center}
\caption{From overcrossings to dots}
\label{fig-crossing-to-dots}
\end{figure}

\begin{figure}[h]
\begin{center}
\rsdraw{.45}{.9}{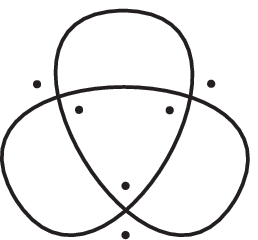} \hspace{1.6cm} \rsdraw{.45}{.9}{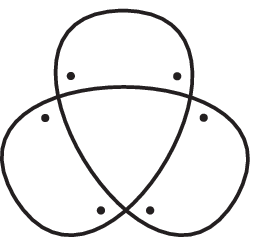}
  \end{center}
\caption{A left-handed trefoil and a right-handed trefoil}
\label{fig-trefoils}
\end{figure}

  Under the passage to  the dotted diagrams, the 
 Reidemeister moves $\Omega_1^\pm, \Omega_2, \Omega_3$  turn respectively into the  moves $M_1^\pm, M_2, M_3$ in Figure~4. The Reidemeister theorem  implies that two dotted  link diagrams  in the plane represent isotopic links in  $\Bbb R^3$ if and only if these  diagrams can be related by  a finite sequence of ambient isotopies in   the plane, the  moves $M_1^\pm, M_2, M_3$, and the inverse moves. 

\begin{figure}[h]
\begin{center}
\psfrag{X}[Bc][Bc]{\scalebox{1}{$M_1^-$}}
\psfrag{A}[Bc][Bc]{\scalebox{1}{$M_1^+$}}
\rsdraw{.45}{.9}{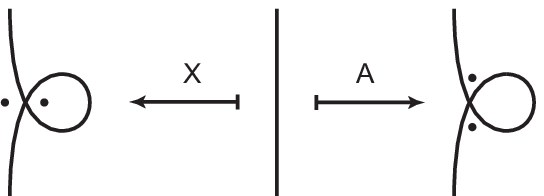}\\[1.5em]
\psfrag{D}[Bc][Bc]{\scalebox{1}{$M_2$}}
\psfrag{T}[Bc][Bc]{\scalebox{1}{$M_3$}}
\rsdraw{.45}{.9}{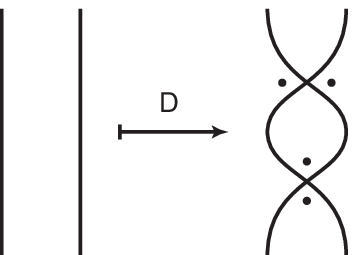}\hspace{1.6cm} \rsdraw{.45}{.9}{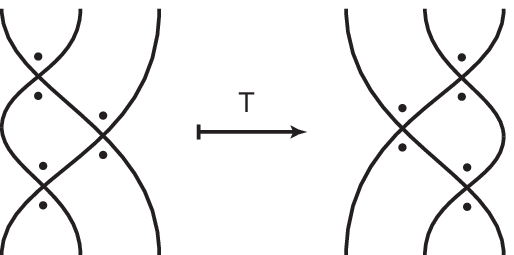}\\
  \end{center}
\caption{The dotted Reidemeister moves}
\label{fig-Reidemeister}
\end{figure}

  The standard procedure of taking the mirror image of a link goes by replacing overpass by underpass at each crossing. In the language of dotted diagrams this amounts to replacing at every crossing  the given pair of dots  with  the \emph{opposite pair} of dots lying in the other two areas  adjacent to the crossing. 
  This yields a dotted diagram of the mirror image of the link.
 For example,   the standard dotted  diagram of the left-handed trefoil  yields in this  way the dotted diagram of  the right-handed trefoil, see Figure~3. It  is curious to note that one of these diagrams has two regions containing 3 dots while   the other diagram has three regions  containing 2 dots.

 \subsection{Remarks}\label{remqsbasics}  1.  Any dotted link diagram~$D$  in $\Bbb R^2$ gives rise to another dotted link diagram $D'$   in $\Bbb R^2$  by reflecting~$D$ (together with the dots) with respect to a straight line  $\ell\subset \Bbb R^2$.   The  links  $L, L'\subset \Bbb R^3$ presented  by $D, D'$ are isotopic: rotating~$L$     around the  line  $\ell\subset\Bbb R^2\subset \Bbb R^3$ to $180^\circ$ we obtain~$L'$.
 As a consequence,  the diagrams $D,D'$ can be related by  a finite sequence of ambient isotopies in  $\Bbb R^2$, the  moves $M_1^\pm, M_2, M_3$, and the inverse moves. 
 
 2. Dotted diagrams and the moves in  Figure~4 can be  also used to represent braids and 
tangles. For example, the dotted diagram in Figure~2 represents the elementary
braid $\sigma_1$ on two strings. The mirror image of this  diagram represents the inverse braid 
 $\sigma_1^{-1}$. The equality  $\sigma_1 \sigma_1^{-1}=1$ is obtained  by the move  $M_2$. The braid identity $\sigma_1\sigma_2\sigma_1 =\sigma_2\sigma_1\sigma_2$ is obtained  by the move  $M_3$.

  \section{Graphs,  2-complexes, and curves}\label{coordinate2}\label{AMdddfT0}  
   
  We fix here our terminology concerning  graphs and curves in 2-complexes.
  
\subsection{Graphs}\label{coordinateddd2w} A \emph{finite graph} is formed by a finite  set of vertices and a finite  set of connecting them edges.  Different edges  may have the same endpoints. We allow  edges  to be loops, i.e., to have both ends in the same vertex.  The \emph{degree} of a vertex is the number of the adjacent edges (the loops  are counted twice). By abuse of the language,  the underlying topological spaces of finite graphs are also called finite graphs.

The \emph{cylinder} over a graph~$\Gamma$ is the topological space $\Gamma\times [0,1]$. The \emph{cone} over~$\Gamma$ is the    topological space $C(\Gamma)$  obtained from the cylinder $\Gamma\times [0,1]$  by contracting $\Gamma \times \{0\}$ to a point.  This point of $C(\Gamma)$  is called the  \emph{cone point}.

The graph  formed by  two vertices and $n\geq 1$ connecting them edges (no loops) will be denoted  $\theta_n$. In particular, $\theta_1$ is a segment, $\theta_2$ is a circle, and $\theta_3$ is a circle with  a diameter.


\subsection{Two-complexes}\label{coordinatedddb2w}  By a \emph{2-complex} we will mean a compact  Hausdorff topological space~$X$ such that each point $x\in X$ has a closed neighborhood  homeomorphic to the  cone $C(\Gamma_x)$ over a  finite graph $\Gamma_x$.  It is understood here that the homeomorphism in question must carry~$x$ to the cone point of  $C(\Gamma_x)$. We distinguish four types of points $x\in X$. If   $\Gamma_x$ is homeomorphic to the circle, then~$x$ is   a \emph{generic point} of~$X$. If   $\Gamma_x$ is homeomorphic to the segment, then~$x$ is   a \emph{boundary point} of~$X$.
If $\Gamma_x$  is homeomorphic  to the graph $\theta_n$ with $n\geq 3$ then $x$ is  a \emph{ridge point of~$X$}.      In all other cases,  $x$ is  a \emph{singular point} of~$X$. 


The generic points of a 2-complex~$X$  form a surface  which we
  call   the \emph{interior} of~$X$ and denote $\Int(X)$. 
The boundary  points of~$X$  form a 1-dimensional manifold denoted $\partial X$. Clearly, the union  $\Int(X)\cup \partial X$ is a surface with boundary~$\partial X$.  The ridge points of~$X$ form a 1-dimensional 
 manifold  called the    \emph{ridge of~$X$} and denoted $\ridge (X)$. Near any ridge point~$x$, the space~$X$ is homeomorphic to the union
  of $\geq 3$ half-planes in $\Bbb R^3$ having the same boundary line. The images of these half-planes in~$X$ are the \emph{branches  of~$X$ at~$x$} or  the \emph{branches  of~$X$ adjacent to~$x$}.

\subsection{Closed curves}\label{sgg21} A  \emph{closed curve} in a 2-complex~$X$ is the image of a continuous map $S^1\to X$.  A   closed curve $\ell   \subset X$ is \emph{generic} if 
 
 - (a) $\ell\subset \Int(X) \cup \ridge(X)$;
 
 - (b) the crossings of~$\ell$  with itself are double transversal crossings in $\Int(X)$ and are finite in number;
 
 - (c)  if~$\ell$  contains a  ridge point~$x$, then near~$x$ this curve  is an    embedded arc which lies in the union of two different  branches of~$X$  at~$x$ and which meets the 1-manifold   $\ridge (X)$  transversally at~$x$.

 A  finite family  of closed curves in~$X$ is \emph{generic} if 
 these curves  are generic and 
 their crossings are double transversal crossings in $\Int(X)$ and are finite in number.

\subsection{Examples}\label{codb2w} 

1. A compact surface is a 2-complex having no ridge points and no singular points.

2. The  underlying topological space of a finite 2-dimensional simplicial complex is a 2-complex.

3. For a  finite   graph $\Gamma$,  the cylinder $\Gamma\times [0,1]$ and the cone $C(\Gamma)$ are 2-complexes.
 If $\Gamma$ has no isolated vertices then $\Gamma\times [0,1]$ has  $2N$ singular points where~$N$ is the number of vertices of~$\Gamma$ of degree $\geq 3$. For $N\geq 1 $,   the cone $C(\Gamma)$ has $N+1$ singular points.


  \section{Knots and links in 2-complexes}\label{s2}  
   
For the rest of the paper, we fix  a  connected 2-complex~$X$.

\subsection{Knots and links in~$X$}\label{s21} 
 A \emph{link}  in the 2-complex~$X$ is a  finite generic family of 
closed curves in~$X$ which is \emph{dotted} in the sense that every  crossing of these curves  is provided with two dots   in the opposite adjacent areas in $\Int(X)$.
 Each closed curve in the link  together with the dots at its self-crossings  is  a \emph{component} of this link.  A link  having only one component is  a \emph{knot}. 
 
 The \emph{mirror image} $L^{\mi}$ of a link~$L$ in~$X$ is  obtained by replacing  at every crossing of~$L$ the given pair of dots with the \emph{opposite pair} of dots lying in the other two adjacent areas at this crossing.

  
\subsection{Isotopies}\label{s22}  Besides the dotted Reidemeister moves $M^{\pm}_1, M_2$, $M_3$  in  $\Int(X)$ we need  moves which  
 involve the  ridge   and the singular points of~$X$.   In  Figures~5 and~6 the horizontal line represents the ridge of~$X$ (not a part of the link). 
 The  moves $M_4$ and $ M_5^{\pm}$ shown  in these figures push the link   across the ridge. 
  These  moves modify  the link   in the union of two adjacent branches.

\begin{figure}[h]
\begin{center}
\psfrag{T}[Bc][Bc]{\scalebox{1}{$M_4$}}
\rsdraw{.45}{.9}{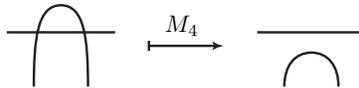}
\end{center}
\caption{The move $M_4$}
\label{fig-move-M4}
\end{figure}

\begin{figure}[h]
\begin{center}
\psfrag{A}[Bc][Bc]{\scalebox{1}{$M_5^+$}}
\rsdraw{.45}{.9}{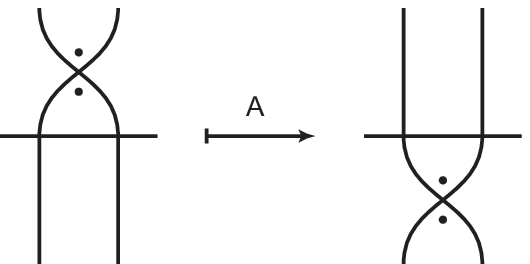}\\[1.5em]
\psfrag{B}[Bc][Bc]{\scalebox{1}{$M_5^-$}}
\rsdraw{.45}{.9}{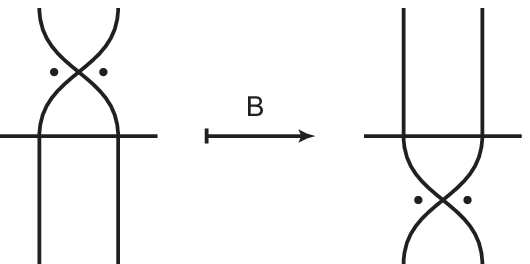}
\end{center}
\caption{The moves $M_5^+$ et $M_5^-$}
\label{fig-moves-M5}
\end{figure}

  The move $M_6$  modifies  the   link    in the union of three adjacent branches, see Figure~7.

\begin{figure}[h]
\begin{center}
\psfrag{T}[Bc][Bc]{\scalebox{1}{$M_6$}}
\rsdraw{.45}{.9}{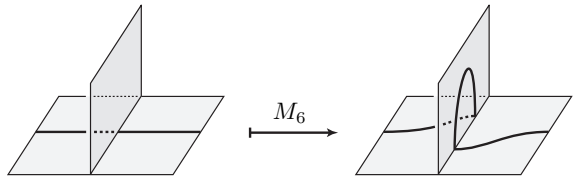}
\end{center}
\caption{The move $M_6$}
\label{fig-move-M6}
\end{figure}

 Finally, the move $M_7$ pushes the link across a singular point $x\in X$ as follows. Let 
$ C(\Gamma_x)\subset X$ be a cone neighborhood  of~$x$ and let $S\subset \Gamma_x$ be a circle embedded in the graph~$\Gamma_x$. Then the cone $D=C(S)\subset C(\Gamma_x)$ is a disc  centered at~$x$. If a  link  in~$X$ meets~$ D$    along an embedded arc with endpoints in $\partial D=S$, then the move $M_7$  replaces this arc with any other embedded arc in $ D\setminus \{x\}$  which  has the same endpoints and is transversal to $\ridge (X) \cap D$.  Note that $\ridge (X) \cap D$ is a union of several
radii of the disc~$D$.

Two links in~$X$ are \emph{isotopic} if
they can be transformed into each other  using  ambient isotopies in~$X$, the moves $M^{\pm}_1, M_2 - M_4, M_5^{\pm}, M_6, M_7$,
 and the inverse moves. We will call these moves \emph{basic moves}. Note that if two  links are isotopic then so are their mirror images. 
 
 The  theory of links in~$X$
 studies  isotopy classes of  links in~$X$.

   \subsection{Local links}\label{s2cv3}  Classical links in $\Bbb R^3$ can be drawn in~$X$ as follows. 
   Pick an embedding $f:\Bbb R^2\hookrightarrow \Int(X)$. Given a link  $L\subset  \Bbb R^3$,   present it by  a dotted diagram in $\Bbb R^2$ and transport this diagram (together with the dots) along~$f$. This yields  a link,   $L_f$, in~$X$ called the \emph{local link modeled on~$L$}.  It is clear  that the isotopy class of $L_f$ depends only on~$L$ and the isotopy class of~$f$.  Moreover, when~$f$ is composed with a reflection of the plane, the isotopy class of $L_f$ is preserved, as  follows from
 Remark~\ref{remqsbasics}.1.  Consequently, the isotopy class of $L_f$   depends only on~$L$ and the connected component of the surface $\Int(X)$ containing the image of~$f$.
 We will say that the link $L_f$ is obtained by drawing~$L$  in that component.
 
 We briefly discuss the role of the   component of $\Int (X)$  above. We say that two  components of  $\Int(X)$ are \emph{adjacent} if they both are adjacent to a certain ridge point of~$X$. Drawing the same  classical link diagram in adjacent components of $\Int (X)$ we  get   isotopic  links in~$X$: the isotopy is obtained by pushing the diagram across the ridge via a sequence of moves $M_4, M_5^{\pm}$ and their inverses. Hence, if any two components of  $\Int(X)$ can be included in a sequence of  components of  $\Int(X)$ in which every two consecutive components are adjacent, then  all links in~$X$ obtained by drawing the same classical link  are isotopic. This property  holds for example when~$X$ is the 2-skeleton of a triangulation of a compact connected 3-manifold.

  \subsection{A relation to links in 3-manifolds}\label{s2dddcv3} Suppose that  our 2-complex~$X$ is embedded into an oriented 3-manifold~$N$. (Note for the record that not all 2-complexes embed in 3-manifolds). We show how to transform any link~$L$ in~$X$ into a link in~$N$. Of course,  the set $L\subset X \subset N$   is a 4-valent graph whose vertices are the crossings of~$L$. 
  We  modify~$L$ near a crossing~$c$  as follows. Recall that $c\in \Int (X)$. Pick a small disk $D\subset \Int (X)$ centered at~$c$ and such that $L\cap D$ is the union of two   diameters of~$D$. These diameters  split~$D$ into four sectors and, by the definition of~$L$, two opposite sectors carry dots. Pick any  orientation~$\omega$ of~$D$. Among the two diameters forming $L\cap D$ we let $d_\omega$ be the one  such that rotating  $d_\omega$ around the center~$c$ of~$D$ in the direction determined by~$\omega$ we first hit the dotted sectors and then the undotted sectors.  Pick a normal vector field $\nu_\omega$ on~$D$ in~$N$ such that the pair $(\omega, \nu_\omega)$ determines the given orientation of~$N$.  Now, we slightly push  $d_\omega\subset L$  along $\nu_\omega$. This separates two branches of~$L$ at~$c$. Moreover, the result of this procedure considered up to isotopy in~$N$  does not depend on~$\omega$. Indeed, under the opposite choice $-\omega$  of the orientation  in~$D$, we obtain the other diameter $d_{-\omega} \neq d_\omega$ and the opposite normal vector field $\nu_{-\omega}=-\nu_\omega$.  Pushing   $d_{-\omega}$  along $\nu_{-\omega}$,
we separate the branches of~$L$ at~$c$ in the same way as above. 

Applying the described procedure at all crossings of~$L$ we transform~$L$ into a link in~$N$.
Its isotopy class  in~$N$ is preserved under the basic moves on~$L$. This is obvious for the moves $ M_4 , M_5^{\pm}, M_6,  M_7$ and standard for  $M^{\pm}_1, M_2, M_3$. (To see it, we can use equivalence between dotted link diagrams in the plane
and  traditional link diagrams). In this way, each isotopy class of links in  $X\subset N$ determines an isotopy class of links in~$N$.

Suppose now that  all  components of $N  \setminus X $ are homeomorphic to a 3-ball. (This is the case, for example, when~$X$ is 
 the 2-skeleton   of a triangulation of~$N$ or  the 2-skeleton   of a CW-decomposition of~$N$ dual to a triangulation of~$N$). Let $\mathcal F$ be the above-defined  mapping from the set of  isotopy classes of links in~$X$  to the set of isotopy classes of links in~$N$. The  mapping~$\mathcal F$ is surjective as any link in~$N$ can be deformed into a position near~$X$ and then drawn on $X$.  It seems plausible that the mapping~$\mathcal F$ is  bijective. Indeed, any link isotopy in~$N$ can be pushed away from the centers of the 3-balls forming $N\setminus X$ to get an isotopy  proceeding near~$X$. The latter link isotopy may be \lq\lq drawn on~$X$\rq\rq, i.e.,  expanded as a composition of ambient isotopies in~$X$
and the basic moves on links in~$X$. Of course, this argument should be carefully verified as the bijectivity of~$\mathcal F$ would have serious consequences. If true, it  would allow  to use isotopy invariants of links in~$X$ (including  those defined in the next sections) to obtain new isotopy invariants of links in~$N$. 
Those invariants would certainly depend on the choice of~$X$. For example, given
  a link $L\subset N$, we can take for~$X$   the 2-skeleton of a triangulation of~$N $ which is so small   that~$L$ is a union of disjoint   closed curves  embedded in~$X$.  Then all invariants of $L\subset X$ defined below are equal to zero. This is in general not the case, for instance, for $X=S^2\subset N=S^3$.


 \subsection{Remark}\label{s2dddc222v3}    As in Section~\ref{dottedbasics},   any traditional link diagram in  a disc $D\subset \Bbb R^2$. transforms  in  a dotted  link  diagram in~$D$. This yields  an equivalence between the theory of links in the 3-manifold $D\times \Bbb R $ and the theory of  links in~$D$. 
 More generally,  for any compact oriented surface~$\Sigma$, the theory of links in the 3-manifold $\Sigma \times \Bbb R$ is equivalent to the theory of links in~$\Sigma$.

  \section{The   linking number and the writhe}\label{s2+}  
 
 We introduce two  numerical characteristics of  links in the  2-complex~$X$, the linking number and the writhe.

  \subsection{Oriented links}\label{s23}  \emph{Oriented links}
in~$X$ are  links in~$X$ with oriented (directed) components.
  \emph{Oriented (basic) moves} $M^{\pm}_1, M_2 - M_4, M_5^{\pm}, M_6, M_7$ on oriented links are the basic moves as  above 
 keeping orientations of all components.  Two oriented  links in~$X$ are \emph{isotopic} if
one can be transformed into the other   using these moves, their inverses, and ambient isotopies in~$X$.     For any oriented link~$L$ in~$X$ we denote by $-L$ the same link with opposite orientation of the components. 
 
 Note that each  Reidemeister move on classical  link diagrams  has several oriented versions. The minimal  set of oriented Reidemeister moves  was described by M. Polyak \cite{Po}. His moves - translated in the language of dotted diagrams -  together with the oriented moves $ M_4, M_5^{\pm}, M_6, M_7$ are sufficient to study isotopies of oriented links in 2-complexes.

   \subsection{The linking number}\label{s2+1}  The linking number  of an   oriented  2-component link in $\Bbb R^3$ counts the number of times that one component winds around the other. This  number can be computed from a  diagram of the link as  a sum of the crossing signs. The signs of the  crossings  can be defined for  oriented links in~$X$:  to a crossing~$c$   we attribute a sign $\sign(c)=\pm 1$ as  in Figure~8. If~$L$ is a 2-component oriented link in~$X$ then we define its linking number  by   \begin{equation}\label{lnre}   
\lk (L)= \sum_c \sign(c) \in \Bbb Z\end{equation} 
  where~$c$ runs over all crossings of the components   of~$L$ with each other. It is easy  to check that    $\lk (L)$  is preserved under
  the oriented basic moves  on~$L$. Therefore  $\lk (L)$ is an  isotopy invariant of~$L$. 

\begin{figure}[h]
\begin{center}
\psfrag{A}[Bl][Bl]{\scalebox{1.111}{$+1$}}
\psfrag{B}[Bl][Bl]{\scalebox{1.111}{$-1$}}
\rsdraw{.45}{.9}{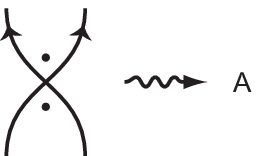}  \hspace{1.6cm}  \rsdraw{.45}{.9}{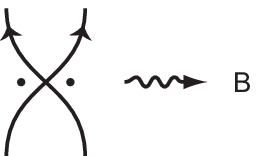}
\end{center}
\caption{The sign of a crossing}
\label{fig-sign}
\end{figure}

 We list a few  properties of  $\lk (L)$. First, $\lk (L)$ is congruent mod 2 to the number of crossings of different components   of~$L$.  Second,   $\lk (-L)=\lk (L)$ and  $\lk (L^{\mi})=-\lk (L)$.  Third, if~$L$ is a local link in~$X$ derived from an oriented 2-component link $\ell=\ell_1 \cup \ell_2 \subset \Bbb R^3$, then    \begin{equation}\label{lnrm}    \lk(L)=\lk(\ell_1, \ell_2)+\lk(\ell_2, \ell_1) =2\lk(\ell_1, \ell_2)=2\lk(\ell_2, \ell_1)  \end{equation}  where $\lk(\ell_1, \ell_2)$ is the usual linking number of the oriented knots $\ell_1, \ell_2$ in $\Bbb R^3$. Indeed, present~$\ell$ by a  traditional planar link diagram and take the associated dotted  diagram. The sum   \eqref{lnre} runs over all crossings of $\ell_1$ with $\ell_2$. Those  crossings where $\ell_1$ lies above $\ell_2$ contribute $\lk(\ell_1, \ell_2)$  and those  crossings where $\ell_2$ lies above $\ell_1$ contribute $\lk(\ell_2, \ell_1)$. This gives \eqref{lnrm}. As a consequence, if $\lk(L)$ is odd then~$L$ is not isotopic to a local link.
  
    \subsection{The writhe}\label{s2+1v}  It is clear from the definitions that the opposite directions on a knot~$K$ in~$X$ give rise to the same signs at all crossings  of~$K$. We define the \emph{writhe $\wri(K)\in \Bbb Z$} of~$K$  to be the sum of the signs of all  crossings of~$K$.  The writhe increases by 1 under the move $M_1^+$,  decreases by 1 under the move $M_1^-$, and  is preserved under
  all  other basic moves. The writhe is \emph{not} an isotopy invariant of knots. 
  
  For any  link~$L$ in~$X$ we define the \emph{unoriented  writhe} $\wri(L)\in \Bbb Z$ as the sum of the writhes of the components of~$L$. 
For an oriented link $L$ in~$X$ we define the \emph{oriented writhe} by 
  \begin{equation}\label{lnr}   
\Wri (L)= \sum_c \sign(c)  \in \Bbb Z \end{equation} 
  where~$c$ runs over all crossings  of~$L$. Clearly, $\Wri (-L)=\Wri (L)$ and  $\Wri (L^{\mi})=-\Wri (L)$. If $L_1,..., L_n$ are the components of~$L$, then 
  $$\Wri(L)=\wri(L)+\sum_{1\leq i<j\leq n} \lk(L_i,L_j).$$
  Both  $\wri(L)$ and $\Wri(L)$   increase by 1 under the move $M_1^+$ on~$L$, decrease by 1 under the move $M_1^-$, and are preserved under
  all  other basic moves. Both these writhes are \emph{not} isotopy invariant.


    \subsection{Examples}\label{s2+22}  1. Pick a point $p\in  S^1$ and consider the 2-component link  in the torus $S^1\times S^1$
    formed by the curves $S^1\times \{p\}, \{p\}\times S^1$, and a choice of dots in their crossing  $(p, p)$.
    For an appropriate  orientation of these curves, we get an oriented link in  the torus with linking number $1$. Inverting  orientation of one of the curves we get an oriented link with linking number $-1$. 
    
    
    2. Let $\mathcal M$ be the M\"{o}bius band represented in Figure~9 by a square whose vertical sides are glued to each other via a homeomorphism carrying the points $A, B$ respectively to $A', B'$.  Let $L$ be the  oriented  2-component  link in~$\mathcal M$ whose components are parametrized by the skew segments  in Figure~9. Then $\lk(L)=1$ and $\lk (L^{\mi})=-1$. As a corollary, the links~$L$ and $L^{\mi}$ are not isotopic.

\begin{figure}[h]
\begin{center}
\psfrag{A}[Br][Br]{\scalebox{1}{$B$}}
\psfrag{K}[Br][Br]{\scalebox{1}{$A$}}
\psfrag{B}[Bl][Bl]{\scalebox{1}{$A'$}}
\psfrag{D}[Bl][Bl]{\scalebox{1}{$B'$}}
\rsdraw{.45}{.9}{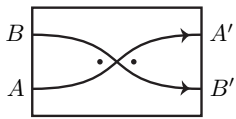}
\end{center}
\caption{A link in a M\"{o}bius band}
\label{fig-moebius}
\end{figure}

3. Let $\mathcal A$ be  the annulus represented in Figure~10 by a square whose vertical sides are glued to each other via a homeomorphism carrying the points $A, B$ respectively to $A', B'$.  Let $L$ be the  oriented 2-component  link in $\mathcal A$ whose components are parametrized by the curves  in Figure~10.  Then $\lk(L)=2$ and $\lk (L^{\mi})=-2$. Thus, the links~$L$, $L^{\mi}$  are not isotopic.

\begin{figure}[h]
\begin{center}
\psfrag{A}[Br][Br]{\scalebox{1}{$B$}}
\psfrag{K}[Br][Br]{\scalebox{1}{$A$}}
\psfrag{B}[Bl][Bl]{\scalebox{1}{$B'$}}
\psfrag{D}[Bl][Bl]{\scalebox{1}{$A'$}}
\rsdraw{.45}{.9}{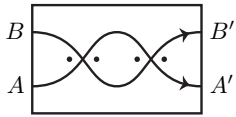}
\end{center}
\caption{A link in an annulus}
\label{fig-moebius}
\end{figure}

4.  Consider the 2-complex $\theta_3\times I$ where $I=[0,1]$ and $\theta_3$ is the graph with two vertices $a,b$ and three connecting them edges $e_0, e_1, e_2$. For any integer $n\geq 0$, we define an oriented 2-component  link $L_n$  in   $\theta_3\times I$. One component of $L_n$ is formed by  an arc leading  from $(a,1/3)$ to $(b,1/3)$ in the square $e_0\times I$ and an arc  going back from  $(b,1/3)$ to $(a,1/3)$   in  the square  $e_1\times I$.   Another component of~$L_n$ is formed by an arc leading from $(a,2/3)$ to $(b,2/3)$ in   $e_0\times I$ and an arc going back   from  $(b,2/3)$ to $(a,2/3)$  in    $e_2\times I$.       We arrange our arcs in $e_0\times I$ so that they meet  transversally   in $2n$ points. We provide these $2n$ crossings with dots  so that  they all have the sign $ +1$. This  defines a 2-component oriented link~$L_n$ in $\theta_3\times I$ with $\lk(L_n)=2n$. Of course, $\lk (L_n^{\mi})=-2n$. As a corollary, the links $\{L_n, L_n^{\mi}\}_{n\geq 1}$ 
 are not isotopic to each other and to $L_0=L_0^{\mi}$.

  \section{Futher invariants of links}\label{s2h+}  
 
 We discuss
   analogues of  linking numbers which take into account  homotopy classes of loops. In this section,~$\Pi$ is  the set of free homotopy classes of loops  in the 2-complex~$X$ (we identify~$\Pi$  with the set of conjugacy classes in  $\pi_1(X)$).  We let $1\in \Pi$ be the  homotopy class of  contractible loops and let $\Bbb Z \Pi$ be the free abelian groups with basis~$\Pi$. 

   \subsection{The   linking class}\label{s2gfg} We define  isotopy 
   invariants of oriented links with values in  $\Bbb Z \Pi$. One  obvious invariant  is the  formal sum of the  free homotopy classes of  link components. This sum, however,  is of little interest to us as it does not depend on the knotting/linking data (the dots). We  define  a subtler  invariant  of 2-component links refining the linking number. Its definition uses
  smoothings of the crossings of oriented curves as in  Figure 11.  (These smoothings do not depend on dots). 

\begin{figure}[h]
\begin{center}
\rsdraw{.45}{.9}{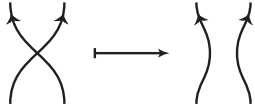}
\end{center}
\caption{The smoothing of a crossing}
\label{fig-smoothing-crossing}
\end{figure}

   Let~$L$ be an oriented 2-component link in~$X$.  The smoothing of~$L$ at a crossing~$c$ of different components of~$L$,    yields a loop in~$X$. Let $\langle c \rangle\in \Pi$  be its free homotopy class. Set
  \begin{equation}\label{lnr+}   
\LK (L)= \sum_c \sign(c)\langle c \rangle \in \Bbb Z \Pi\end{equation} 
  where~$c$ runs over all crossings  of the components of~$L$ with each other. It is easy  to check that $\LK (L)$ is preserved under
  the oriented basic moves  on~$L$, and therefore is an isotopy  invariant of~$L$. We call it the \emph{linking class} of~$L$. Clearly,  the additive map $\Bbb Z \Pi\to \Bbb Z, \Pi\mapsto 1$ carries  $\LK(L)$ into $\lk(L)$. Also $\LK (L^{\mi})=-\LK (L)$ and  $\LK(-L)=\overline {\LK(L)}$ where the overline denotes the involution in $\Bbb Z \Pi$ induced by the inversion in $\pi_1 (X)$. The reader 
can easily compute the  linking classes of the links  in Section~\ref{s2+22}.


  \subsection{The   colinking class}\label{s2nmgfg} In analogy with the linking class, we introduce  an  isotopy 
   invariant of oriented knots in~$X$ with values in the abelian group  $\Bbb Z \Pi \otimes_{\Bbb Z} \Bbb Z \Pi$. 
Namely, consider an  oriented knot~$K$ in~$X$. Let $k\in \Pi$ be the free homotopy class of the underlying loop of~$K$. 
 The smoothing of~$K$ at any crossing~$c$  yields an oriented 2-component link  in~$X$. Let $k_1^c, k_2^c\in \Pi$ be the free homotopy classes of the components of this link enumerated in an arbitrary order.
Set
  \begin{equation}\label{lnr+co}   
\co (K)= \sum_c \sign(c) (k_1^c\otimes k_2^c + k_2^c\otimes k_1^c - k\otimes 1-1\otimes k)   \in  \Bbb Z \Pi \otimes_{\Bbb Z} \Bbb Z \Pi\end{equation}  
where~$c$ runs over all crossings  of~$K$. 
Equivalently,
$$
\co (K)=\sum_c \sign(c) (k_1^c\otimes k_2^c + k_2^c\otimes k_1^c ) -  \wri (K) (k\otimes 1+1\otimes k)  .$$ It is easy  to check that $\co (K)$ is preserved under
  all oriented basic moves on~$K$, and therefore is an isotopy  invariant of~$K$. We call  $\co (K)$ the \emph{colinking class} of~$K$. Clearly, $\co(K^{\mi})=-\co (K)$ and $\co(-K)=\overline{\co(K)}$ where the overline stands for the tensor square of the involution   in $\Bbb Z \Pi $ induced by the inversion in $\pi_1 (X)$.
  Also, for any local knot~$K$, we have $\co(K)=0$.
  
    \subsection{Example}\label{s2nmgexeg} We compute the colinking class for a sequence of oriented knots $\{K_n\}_{n\geq 0}$  in the 2-complex $\theta_3\times I$. We use the symbols $a,b, e_0, e_1, e_2$    as in Example~\ref{s2+22}.4. For $n\geq 0$,  the knot~$K_n$ is formed by  four oriented embedded arcs: an arc  from $(a,1/3)$ to $(b,2/3)$ in $e_0\times I$, an arc  from  $(b,2/3)$ to $(a,2/3)$  in  $e_1\times I$, an arc  from  $(a,2/3)$ to $(b,1/3)$  in $e_0\times I$, and finally an arc  from  $(b,1/3)$ to $(a,1/3)$ in  $e_2\times I$.
    We arrange the two arcs in $e_0\times I$ so that they cross  each other transversally  in $2n+1$ points. We provide these crossings with dots  so that all the crossings have sign  $ +1$. This  defines an oriented knot~$K_n$ in $\theta_3\times I$. To calculate $\co(K_n)$, note  that the group $\pi_1(\theta_3\times I)=
    \pi_1(\theta_3)$ is a free group on two generators $u,v$ represented respectively   by the loops $e_0 e_1^{-1}$ and $e_0 e_2^{-1}$ where $e_0, e_1,e_2$ are directed from~$a$ to~$b$. It follows from the definitions that $$ \co (K_n)= (2n+1) (u \otimes v +v \otimes u - k_n\otimes 1- 1\otimes k_n )$$
    where  $k_n\in \Pi$ is the free homotopy class of the underlying loop of~$K_n$. These computations imply that the knot $K_n$ is not isotopic to a knot with less than $2n+1$ crossings. In particular, the knots $K_n, K_m$ are not isotopic for $n\neq m$.

  \section{The bracket polynomial}\label{sbbb2h+}  
 
We extend Louis Kauffman's definition of the bracket polynomial of classical links  to  links in  2-complexes.

   \subsection{Definition}\label{s2nnnnmsgfg} For a link~$L$ in the 2-complex~$X$ we denote by $\# L$  the set of  crossings of~$L$ and let $\cro(L)=\card(\#L)$ be the number of crossings of~$L$. Given a set $C\subset \# L$, we smooth~$L$ at
all crossing  as follows: at the   crossings belonging to~$C$ we proceed as in the left part of Figure~12; at the   crossings not belonging to~$C$ we proceed as in the right part of Figure~12. This turns $L$ into a system of  disjoint simple closed curves in~$X$. Let $\vert L, C\vert \geq 1$ be the number of these curves. Set 
\begin{equation}\label{idyvgm3n+} \langle L, C \rangle=   (-A^2-A^{-2})^{\vert L, C \vert -1} A^{2\,{ \card}\, (C)-   \cro (L)} \in \Bbb Z[A, A^{-1}].
\end{equation}
The \emph{bracket polynomial} $\langle L \rangle$ of~$L$ is the sum 
\begin{equation}\label{idyvg3n+} \langle L\rangle=  \sum_{C\subset \#L}  \langle L, C \rangle \in \Bbb Z[A, A^{-1}] \end{equation}
where summation  runs over all subsets~$C$ of $ \#L$.
It is easy  to check that the Laurent polynomial $\langle L \rangle$ is preserved under
  the   (unoriented) basic moves $M_2 - M_4, M_5^{\pm}, M_6, M_7$  on~$L$. The only moves  requiring  a little work  are the  moves 
  $M_2 ,M_3$ in a disc in $\Int(X)$;  to handle them we can switch inside this disc to the  language of over/under-passes and use the classical Kauffman's arguments.

\begin{figure}[h]
\begin{center}
\rsdraw{.45}{.9}{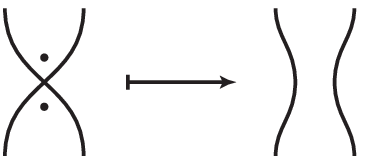}  \hspace{1.6cm}  \rsdraw{.45}{.9}{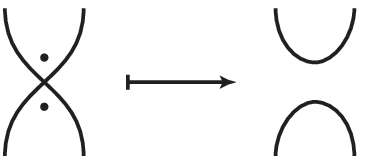}
\end{center}
\caption{Smoothings}
\label{fig-smoothings}
\end{figure}

Under
  the  moves $ M^{\pm}_1$  the bracket polynomial   changes: it is multiplied by $-A^{-3}$ under $M^+_1$
  and  is multiplied by $-A^{3}$ under $M^-_1$.  Nevertheless, one can extract isotopy invariants of links from their bracket polynomials. One such invariant is the span of the bracket polynomial discussed  in the next subsection. Another isotopy invariant is obtained by  combining the bracket polynomial with the writhe.  Namely, given a link~$L$ in~$X$, we define its \emph{normalized bracket polynomial}   
  $ (-A^3)^{-\wri(L)} \langle L \rangle $.
This polynomial is preserved under all basic moves  and is an  isotopy invariant of~$L$. A different normalization of the bracket polynomial can be defined for oriented links: use $\Wri$ instead of $\wri$ in the formula above.


  For local links  our bracket polynomial coincides with Kauffman's bracket polynomial. It therefore can be viewed as a  generalization of the latter to our setting. 
   Substituting $A=t^{-1/4}$ in 
the  normalized bracket polynomial of oriented links  in~$X$ we get a function in~$t$  generalizing the Jones polynomial of  oriented links in $\Bbb R^3$.

   \subsection{An estimate}\label{sbnnmsgfg} The bracket polynomial of classical links can be used to  estimate the number of  crossings in link diagrams. We obtain here a similar result for links in~$X$. We first recall the span of a  polynomial $f \in \Bbb Z[A, A^{-1}]$. If $f\neq 0$, then we 
 expand~$f$  as a linear combination of monomials $  A^m, A^{m+1},  ... ,A^{n}$ for certain integers $m\leq n$ such that both monomials $A^m, A^n$ appear in~$f$ with non-zero coefficients  (possibly, $m=n$). We call~$m$ the \emph{low degree}  and~$n$ the \emph{top degree} of~$f$. The integer $\spun(f)=n-m$ is  the \emph{span} of~$f$. For $f=0$,  set $\spun(f)=0$. 
      
      
      We  say that two links in~$X$ are \emph{disjoint} if the underlying loops of these links do not  meet.  We say that a link in~$X$ is \emph{unsplittable} if it is not a union of  non-empty disjoint links.  For example, all knots in~$X$ are unsplittable. It is clear that every link~$L$ in~$X$ splits uniquely  as a union of several disjoint unsplittable   links. These links are called \emph{split components} of~$L$ and their number is denoted by $sc(L)$.  Note that, in general, the number $sc(L)$ is \emph{not} an isotopy invariant of~$L$.

\begin{theor}\label{1aeee+m+} For any link~$L$ in a 2-complex~$X$, we have $$\cro(L)\geq 1-sc(L) +\spun(\langle L \rangle)/4.$$ In particular, for any unsplittable link~$L$ in~$X$, we have $\cro(L)\geq \spun(\langle L \rangle)/4$. \end{theor}

\begin{proof}  Let $r=\cro(L)$ be the number of crossings of~$L$ and let $s=sc(L)$ be the number of split components of~$L$. We prove first that
\begin{equation}\label{dualst}  \vert L, \#L \vert +\vert L, \emptyset \vert \leq r+2s. \end{equation} We will use the same method as in  the proof of Lemma 1 in \cite{Tu}. Denote by~$\Gamma$ be the union of the underlying loops of~$L$. This is a 4-valent graph embedded in~$X$ and consisting of generic points of~$X$ and  a finite set of ridge points of~$X$. We thicken~$\Gamma$ to a surface which is formed by square neighborhoods of  the vertices of~$\Gamma$ (= crossings of~$L$) in $\Int(X)$ and by narrow bands connecting the sides of these squares  and obtained  by thickening the edges of~$\Gamma$. There are two subtleties. First, if an edge of~$\Gamma$ crosses a ridge point~$x$ then  the thickening of this edge near~$x$ proceeds  in the union of two branches of~$X$ at~$x$ containing a part of~$\Gamma$ near~$x$. Second,   if a  band connects the sides $a,b$ of  square neighborhoods of  vertices of~$\Gamma$, then the dotted vertex of~$a$ should lie on the same long side of this band as the dotted vertex of~$b$; otherwise, the band receives an additional half-twist to ensure this condition (in the latter  case the band does not lie in~$X$). This construction yields a surface $\Sigma\supset \Gamma$  such that the 1-manifold $\partial \Sigma$ has $\vert L, \#L \vert $ components  containing no dots and
$ \vert L, \emptyset \vert$  components whose union contains all dots of $L$.  Thus, $\partial \Sigma$ has $  \vert L, \#L \vert +\vert L, \emptyset \vert $ components. The surface~$\Sigma$ itself has $s=sc(L)$ components.

 Let $b_i$ denote the $i$-th Betti number of a space  with coefficients in $\Bbb Z/2\Bbb Z$.
We have  $b_i(\Sigma)=b_i(\Gamma)$ for all~$i$ as the graph~$\Gamma$ is a deformation retract of~$\Sigma$. Since~$\Gamma$ has $r=\cro(L)$ vertices and they all have degree 4, this graph has $2r$ edges. Thus, $$\chi(\Sigma)=\chi(\Gamma)=r-2r=-r\quad {\rm{and}}
\quad b_1(\Sigma)=b_0(\Sigma)-\chi(\Sigma)=s+r.$$ The exact homology sequence of the pair $(\Sigma, \partial \Sigma)$ and the Poincar\'e duality yield
$$ \vert L, \#L \vert +\vert L, \emptyset \vert =b_0(\partial \Sigma)\leq b_0(\Sigma)+b_1(\Sigma, \partial \Sigma)$$
$$=b_0(\Sigma)+b^1(\Sigma)
=b_0(\Sigma)+b_1(\Sigma)=r+2s.$$
This gives \eqref{dualst}.

 Observe now that if a set $C\subset\#L$ is obtained by removing one element from~$\#L$, then the smoothings of~$L$ determined by the sets $\#L$ and~$C$ differ only at one crossing of~$L$. Consequently, $\vert L, C \vert\leq \vert L, \#L \vert+1$.  
Proceeding by induction, we  obtain  that if a set~$C$  is obtained by removing $k\geq 1$ elements  from the set of~$r$ elements~$\#L$, then $\vert L, C \vert\leq \vert L, \#L \vert+k$ and $\card (C)=r-k$. The top degree of the polynomial $\langle L, C \rangle$ is equal to
$$2(\vert L, C \vert -1+ { \card}\, C)-   r=2\vert L, C \vert -2-2k+   r$$ and  is bounded from above by  $ 2\vert L, \#L \vert -2+  r$. The top degree of the polynomial
$\langle L \rangle=\sum_C \langle L, C \rangle$ is bounded from above by  the same number.  

Similarly, if a set  $C\subset \# L$ has~$l$ elements then $\vert L,  C \vert \leq  \vert L,  \emptyset \vert+l$. Therefore the low degree $2-2 \vert L,  C \vert+2l-r$ of the polynomial
$\langle L, C \rangle$  is greater than or equal to  $2-2 \vert L,  \emptyset \vert-r$.  The low degree of the polynomial
$\langle L \rangle=\sum_C \langle L, C \rangle$ is bounded from below by
the same number. We conclude that 
$$\spun(\langle L \rangle)\leq (2\vert L, \#L \vert -2 +   r)-( 2-2 \vert L,\emptyset \vert-r)$$
$$=2\vert L, \#L \vert+2 \vert L,\emptyset \vert +2 r-4.$$
Combining with \eqref{dualst}, we get $\spun(\langle L \rangle)\leq 4r +4s-4$. This is equivalent to the claim of the theorem.
\end{proof}

      \subsection{Homotopy bracket polynomial}\label{s2nnnfcvgmsgfg} In analogy  with Section~\ref{s2h+}, we   define a  homotopy bracket polynomial of links in~$X$. By a \emph{simple system of curves} we shall mean a finite (possibly empty) family of disjoint simple closed curves in~$X$. We will make no difference between such a family and the union of  curves in this family. Two simple systems of curves are \emph{homotopic} if one can be deformed into the other in the class of simple systems of curves. Let $\mathcal X$ be the set of  homotopy classes of simple systems of curves in~$X$. Consider the ring $R=\Bbb Z[A, A^{-1}]$  and let $R\mathcal X$ be the free $R$-module with basis~$\mathcal X$. 
      We   denote by $\overline{R\mathcal X}$ the quotient of the module $R\mathcal X$ by the submodule generated
by the elements of type $s'+(A^2+A^{-2})s$ where $s, s'$ are simple systems of curves in~$X$ such that~$s'$ is obtained from~$s$ by adding a simple closed curve in $X\setminus s$ which is contractible in $X\setminus s$.

      For a link~$L$ in~$X$ and a set $C\subset \# L$, we smooth~$L$ at
all crossing  as  in Section~\ref{s2nnnnmsgfg}. This gives  a simple system of   curves, $s_C$, in~$X$.  The \emph{homotopy bracket}  of~$L$ is the sum 
\begin{equation}\label{idyvgm3n+WW} \langle\langle L \rangle \rangle=     \sum_{C\subset \#L}  A^{2\,{ \card}\, (C)-   \cro (L)} s_C \in \overline{R\mathcal X}
\end{equation}
where summation  runs over all sets $C\subset \#L$.
It is easy  to check that $\langle\langle L \rangle \rangle$ is preserved under
  the   (unoriented) basic moves $M_2 - M_4, M_5^{\pm}, M_6, M_7$  on~$L$.
  Under
  the  move $ M^{\pm}_1$  the homotopy bracket is multiplied by $-A^{\mp 3}$. The  \emph{normalized homotopy bracket}   
  $ (-A^3)^{-\wri(L)}\langle \langle L\rangle \rangle  \in \overline{R\mathcal X} $ is preserved under all basic moves  and is an  isotopy invariant of~$L$

\end{document}